\documentclass[12pt]{amsart}
\usepackage{mathrsfs}
\usepackage{amssymb}
\usepackage{verbatim}
\usepackage{hyperref}
\usepackage{color}
\usepackage{amsfonts}
\usepackage{mathrsfs}
\usepackage{amsmath}
\usepackage{amssymb}

\begin{document}
\newtheorem{theorem}{Theorem}
\newtheorem{proposition}[theorem]{Proposition}
\newtheorem{conjecture}[theorem]{Conjecture}
\def\theconjecture{\unskip}
\newtheorem{corollary}[theorem]{Corollary}
\newtheorem{lemma}[theorem]{Lemma}
\newtheorem{sublemma}[theorem]{Sublemma}
\newtheorem{observation}[theorem]{Observation}
\theoremstyle{definition}
\newtheorem{definition}{Definition}
\newtheorem{notation}[definition]{Notation}
\newtheorem{remark}[definition]{Remark}
\newtheorem{question}[definition]{Question}
\newtheorem{questions}[definition]{Questions}
\newtheorem{example}[definition]{Example}
\newtheorem{problem}[definition]{Problem}
\newtheorem{exercise}[definition]{Exercise}

\numberwithin{theorem}{section} \numberwithin{definition}{section}
\numberwithin{equation}{section}

\def\earrow{{\mathbf e}}
\def\rarrow{{\mathbf r}}
\def\uarrow{{\mathbf u}}
\def\varrow{{\mathbf V}}
\def\tpar{T_{\rm par}}
\def\apar{A_{\rm par}}

\def\reals{{\mathbb R}}
\def\torus{{\mathbb T}}
\def\heis{{\mathbb H}}
\def\integers{{\mathbb Z}}
\def\naturals{{\mathbb N}}
\def\complex{{\mathbb C}\/}
\def\distance{\operatorname{distance}\,}
\def\support{\operatorname{support}\,}
\def\dist{\operatorname{dist}\,}
\def\Span{\operatorname{span}\,}
\def\degree{\operatorname{degree}\,}
\def\kernel{\operatorname{kernel}\,}
\def\dim{\operatorname{dim}\,}
\def\codim{\operatorname{codim}}
\def\trace{\operatorname{trace\,}}
\def\Span{\operatorname{span}\,}
\def\dimension{\operatorname{dimension}\,}
\def\codimension{\operatorname{codimension}\,}
\def\nullspace{\scriptk}
\def\kernel{\operatorname{Ker}}
\def\ZZ{ {\mathbb Z} }
\def\p{\partial}
\def\rp{{ ^{-1} }}
\def\Re{\operatorname{Re\,} }
\def\Im{\operatorname{Im\,} }
\def\ov{\overline}
\def\eps{\varepsilon}
\def\lt{L^2}
\def\diver{\operatorname{div}}
\def\curl{\operatorname{curl}}
\def\etta{\eta}
\newcommand{\norm}[1]{ \|  #1 \|}
\def\expect{\mathbb E}
\def\bull{$\bullet$\ }

\def\xone{x_1}
\def\xtwo{x_2}
\def\xq{x_2+x_1^2}
\newcommand{\abr}[1]{ \langle  #1 \rangle}

\newcommand{\Norm}[1]{ \left\|  #1 \right\| }
\newcommand{\set}[1]{ \left\{ #1 \right\} }
\def\one{\mathbf 1}
\def\whole{\mathbf V}
\newcommand{\modulo}[2]{[#1]_{#2}}
\def \essinf{\mathop{\rm essinf}}
\def\scriptf{{\mathcal F}}
\def\scriptg{{\mathcal G}}
\def\scriptm{{\mathcal M}}
\def\scriptb{{\mathcal B}}
\def\scriptc{{\mathcal C}}
\def\scriptt{{\mathcal T}}
\def\scripti{{\mathcal I}}
\def\scripte{{\mathcal E}}
\def\scriptv{{\mathcal V}}
\def\scriptw{{\mathcal W}}
\def\scriptu{{\mathcal U}}
\def\scriptS{{\mathcal S}}
\def\scripta{{\mathcal A}}
\def\scriptr{{\mathcal R}}
\def\scripto{{\mathcal O}}
\def\scripth{{\mathcal H}}
\def\scriptd{{\mathcal D}}
\def\scriptl{{\mathcal L}}
\def\scriptn{{\mathcal N}}
\def\scriptp{{\mathcal P}}
\def\scriptk{{\mathcal K}}
\def\frakv{{\mathfrak V}}
\def\C{\mathbb{C}}
\def\R{\mathbb{R}}
\def\Rn{{\mathbb{R}^n}}
\def\Sn{{{S}^{n-1}}}
\def\M{\mathbb{M}}
\def\N{\mathbb{N}}
\def\Q{{\mathbb{Q}}}
\def\Z{\mathbb{Z}}
\def\F{\mathcal{F}}
\def\L{\mathcal{L}}
\def\S{\mathcal{S}}
\def\supp{\operatorname{supp}}
\def\dist{\operatorname{dist}}
\def\essi{\operatornamewithlimits{ess\,inf}}
\def\esss{\operatornamewithlimits{ess\,sup}}
\author{Mingming Cao}
\address{Mingming Cao \\
         School of Mathematical Sciences \\
         Beijing Normal University \\
         Laboratory of Mathematics and Complex Systems \\
         Ministry of Education \\
         Beijing 100875 \\
         People's Republic of China}
\email{m.cao@mail.bnu.edu.cn}

\author{Qingying Xue}
\address{Qingying Xue\\
        School of Mathematical Sciences\\
        Beijing Normal University \\
        Laboratory of Mathematics and Complex Systems\\
        Ministry of Education\\
        Beijing 100875\\
        People's Republic of China}
\email{qyxue@bnu.edu.cn}

\thanks{The second author was supported partly by NSFC
(No. 11471041), the Fundamental Research Funds for the Central Universities (NO. 2012CXQT09 and NO. 2014KJJCA10) and NCET-13-0065. \\ \indent Corresponding
author: Qingying Xue\indent Email: qyxue@bnu.edu.cn}

\keywords{Bi-parameter Littlewood-Paley $g_{\lambda}^{*}$-function; Probabilistic methods; Haar functions.}

\date{December 1, 2015.}
\title[Bi-parameter Littlewood-Paley $g_{\lambda}^{*}$-Function]{\textbf{On The Boundedness of Bi-parameter Littlewood-Paley} $g_{\lambda}^{*}$-function}
\maketitle

\begin{abstract}
Let $m,n\ge 1$ and $g_{\lambda_1,\lambda_2}^*$ be the bi-parameter Littlewood-Paley $g_{\lambda}^{*}$-function defined by
\begin{align*}
g_{\lambda_1,\lambda_2}^*(f)(x)
&= \bigg(\iint_{\R^{m+1}_{+}} \Big(\frac{t_2}{t_2 + |x_2 - y_2|}\Big)^{m \lambda_2}
\iint_{\R^{n+1}_{+}} \Big(\frac{t_1}{t_1 + |x_1 - y_1|}\Big)^{n \lambda_1} \\
&\quad\quad \times |\theta_{t_1,t_2} f(y_1,y_2)|^2 \frac{dy_1 dt_1}{t_1^{n+1}} \frac{dy_2 dt_2}{t_2^{m+1}} \bigg)^{1/2}, \quad\quad \quad\quad\lambda_1>1,\quad \lambda_2>1
\end{align*}
where $\theta_{t_1,t_2} f$ is a non-convolution kernel defined on $\mathbb{R}^{m+n}$. In this paper, we showed that the bi-parameter Littlewood-Paley function $g_{\lambda_1,\lambda_2}^*$ was bounded from $L^2(\R^{n+m})$ to $L^2(\R^{n+m})$. This was done by means of probabilistic methods and by using a new averaging identity over good double Whitney regions.
\end{abstract}


\section{Introduction}
\subsection{Background and motivation}
It is well known that $g_{\lambda}^*$-function originated in the work of Littlewood and Paley \cite{LP} in the 1930's. In 1961, Stein \cite{Stein1970} introduced and studied the following higher dimensional ($n \geq 2$) Littlewood-Paley $g_{\lambda}^*$-function:
$$ g_{\lambda}^*(f)(x)=\bigg(\iint_{\R^{n+1}_{+}} \Big(\frac{t}{t+|x-y|}\Big)^{n\lambda}
|\nabla P_t f(y,t)|^2 \frac{dy dt}{t^{n-1}}\bigg)^{1/2},\quad\quad \lambda > 1$$
where $P_t f(y,t)=p_t*f(x)$, $p_t(y)=t^{-n}p(y/t)$ denotes the Poisson kernel and
$\nabla =(\frac{\partial}{\partial y_1},\ldots,\frac{\partial}{\partial y_n},\frac{\partial}{\partial t})$. It plays important roles in harmonic analysis and other fields. With much greater difficulty, Stein \cite{Stein1961} showed that $\big\| g_{\lambda}^*(f) \big\|_{L^p(\Rn)}$ and $\big\| f \big\|_{L^p(\Rn)}$ are equivalent norms for any $1< p < \infty$. Moreover, in \cite{Stein1961}, Stein also proved that if $\lambda > 2$, then $g_{\lambda}^*$ is of weak type $(1,1)$, and is of strong type $(p,p)$ for $1 < p < \infty$. In the same paper, Stein pointed out that weak $(1,1)$ boundedness doesn't holds for $1<\lambda\le 2.$ In 1970, as a replacement of weak $(1,1)$ bounds for $1<\lambda<2$, Fefferman \cite{F} considered the end-point weak $(p,p)$ estimates of $g_{\lambda}^*$-function when $p>1$ and $\lambda=2/p$.

Recently, Cao, Xue ad Li \cite{CLX} gave a characterization of two weight norm inequalities for the classical $g_\lambda^*$-function. The first step of the proof is to reduce the case to good Whitney regions. In addition, the random dyadic grids and martingale differences decomposition are used. The core of the proof is the construction of stopping cubes, which is a modern and effective technique to deal with two weight problems. The stopping cubes were first introduced to handle two weight boundedness of Hilbert transform \cite{Lacey1}, \cite{Lacey2}. Then, some related consequences and applications were given, as demonstrated in \cite{CLX}, \cite{LL} and \cite{LSTS}. Still more recently, Cao and Xue \cite{CX} established a local $Tb$ theorem for the non-homogeneous Littlewood-Paley $g_{\lambda}^{*}$-function with non-convolution type kernels and upper power bound measure $\mu$. It was the first time to investigate $g_\lambda^*$-function in the simultaneous presence of three attributes : local, non-homogeneous and $L^p$-testing condition. 

When it comes to the multi-parameter harmonic analysis, there is a very large existing theory. In 2012, a dyadic representation theorem for bi-parameter singular integrals was presented by Martikainen \cite{M2012} and a new version of $T1$ theorem on the product space was also established. In 2014, Hyt\"{o}nen and Martikainen \cite{HM} proved a non-homogeneous version of $T1$ theorem for certain bi-parameter singular integral operators. Moreover, they discussed
the related non-homogeneous Journ\'{e}¡¯s lemma and product BMO theory with more general type of measures. Still in 2014, a class of bi-parameter kernels and related vertical square functions in the upper half-space were first introduced by Martikainen \cite{M2014}. Using dyadic probabilistic techniques, the author gave a criterion for the $L^2(\R^{n+m})$ boundedness of these square functions. It is worth pointing out that the kernels are assumed to satisfy some estimates, including a natural size condition, a H\"{o}lder estimate and two symmetric mixed H\"{o}lder and size estimates, the mixed Carleson and size conditions, the mixed Carleson and H\"{o}lder estimates and a bi-parameter Carleson condition. Moreover, it should be noted that the bi-parameter Carleson condition is necessary for the square function to be bounded on $L^2(\R^{n+m})$.

Motivated by the above works, in this paper, we keep on studying the Littlewood Paley $g_\lambda^*$-function but in the bi-parameter setting. First, we introduce the definition of the bi-parameter Littlewood Paley $g_\lambda^*$-function.

\begin{definition}\label{definition 1.1}
Let $\lambda_1,\lambda_2 >1$, for any $x=(x_1,x_2) \in \R^{n+m}$, the bi-parameter Littlewood-Paley $g_\lambda^*$-function is defined by
\begin{align*}
g_{\lambda_1,\lambda_2}^*(f)(x)
&:= \bigg(\iint_{\R^{m+1}_{+}} \Big(\frac{t_2}{t_2 + |x_2 - y_2|}\Big)^{m \lambda_2}
\iint_{\R^{n+1}_{+}} \Big(\frac{t_1}{t_1 + |x_1 - y_1|}\Big)^{n \lambda_1} \\
&\quad\quad\quad \times |\theta_{t_1,t_2} f(y_1,y_2)|^2 \frac{dy_1 dt_1}{t_1^{n+1}} \frac{dy_2 dt_2}{t_2^{m+1}} \bigg)^{1/2},
\end{align*}
where
$\theta_{t_1,t_2} f(y_1,y_2) = \iint_{\R^{n+m}} K_{t_1,t_2}(y_1,y_2,z_1,z_2)f(z_1,z_2) dz_1 \ dz_2.$
\end{definition}

Under certain structural assumptions, we will prove the  $L^2(\R^{n+m})$ boundedness of $g_{\lambda_1,\lambda_2}^*$, in other words, the following inequality,
$$
\big\| g_{\lambda_1,\lambda_2}^*(f) \big\|_{L^2(\R^{n+m})} \lesssim \big\| f \big\|_{L^2(\R^{n+m})}.
$$

Compared to the bi-parameter vertical square function, the bi-parameter Littlewood Paley $g_\lambda^*$-function is significantly much more difficult to be dealt with. Actually, in bi-parameter case, additional integrals make most of the corresponding estimates more complicated. We could not use the assumptions in \cite{M2014} directly, since addition terms appears in the Definition \ref{definition 1.1}. In fact, we will use much more weaker conditions than the conditions used in \cite{M2014} (see assumptions in the following subsection).  Unlike the one-parameter case and two-weight case \cite{CLX}, the proof of bi-parameter  $g_\lambda^*$-function does not involve the stopping cubes and martingale differences decomposition. In fact, the decomposition associated with Haar function in $\Rn$ provides a foundation for our analysis. And modern techniques, including probabilistic methods and dyadic analysis, will be used efficiently again. They were first used by Martikainen \cite{M2012} in the study of the bi-parameter Calder\'{o}n-Zygmund integrals and later appeared in \cite{M2014}. For more applications, one can refer \cite{HM}, \cite{Ou}. However, our result is based on a simple new averaging identity over good double Whitney regions.

\subsection{Assumptions and Main result}
To state our main results, we need to give some appropriate assumptions. From now on, we always assume that
$\alpha, \beta>0$. We use, for minor convenience, $\ell^\infty$ metrics on $\Rn$ and $\R^m$.

\vspace{0.3cm}
\noindent\textbf{Assumption 1 (Standard estimates).} The kernel $K_{t_1,t_2}: \R^{n+m} \times \R^{n+m} \rightarrow \C$ is assumed to satisfy the following estimates:
\begin{enumerate}
\item [(1)] Size condition :
$$ |K_{t_1,t_2}(x,y)|
\lesssim \frac{t_1^{\alpha}}{(t_1 + |x_1 - y_1|)^{n+\alpha}} \frac{t_2^{\beta}}{(t_2 + |x_2 - y_2|)^{m+\beta}}.$$
\item [(2)] H\"{o}lder condition :
\begin{align*}
|K_{t_1,t_2}(x,y) -& K_{t_1,t_2}(x,(y_1,y_2')) - K_{t_1,t_2}(x,(y_1',y_2)) + K_{t_1,t_2}(x,y')| \\
& \lesssim \frac{|y_1 - y_1'|^{\alpha}}{(t_1 + |x_1 - y_1|)^{n+\alpha}} \frac{|y_2 - y_2'|^{\beta}}{(t_2 + |x_2 - y_2|)^{m+\beta}},
\end{align*}
whenever $|y_1 -y_1'| < t_1/2$ and $|y_2 - y_2'| < t_2/2$.
\item [(3)] Mixed H\"{o}lder and size conditions :
$$ |K_{t_1,t_2}(x,y) - K_{t_1,t_2}(x,(y_1,y_2'))|
\lesssim \frac{t_1^{\alpha}}{(t_1 + |x_1 - y_1|)^{n+\alpha}} \frac{|y_2 - y_2'|^{\beta}}{(t_2 + |x_2 - y_2|)^{m+\beta}},$$
whenever $|y_2 -y_2'| < t_2/2$ and
$$ |K_{t_1,t_2}(x,y) - K_{t_1,t_2}(x,(y_1',y_2))|
\lesssim \frac{|y_1 - y_1'|^{\alpha}}{(t_1 + |x_1 - y_1|)^{n+\alpha}} \frac{t_2^{\beta}}{(t_2 + |x_2 - y_2|)^{m+\beta}},$$
whenever $|y_1 - y_1'| < t_1/2$.
\end{enumerate}
\noindent\textbf{Assumption 2 (Carleson condition $\times$ Standard estimates).}
If $I \subset \Rn$ is a cube with side length $\ell(I)$, we define the associated Carleson box by $\widehat{I}=I \times (0,\ell(I))$. We assume the following conditions : For every cube $I \subset \Rn$ and $J \subset \R^m$, there holds that
\begin{enumerate}
\item [(1)] Combinations of Carleson and size conditions :
\begin{align*}
\bigg( \iint_{\widehat{I}} \int_{\Rn} \bigg| \int_{I} K_{t_1,t_2}(x-y,z_1,z_2) dz_1 \bigg|^2 \Big(\frac{t_1}{t_1 + |y_1|}\Big)^{n \lambda_1} &\frac{dy_1dx_1 dt_1}{t_1^{n+1}} \bigg)^{\frac{1}2} \\
&\lesssim |I|^{\frac12}\frac{ t_2^{\beta}}{(t_2 + |x_2 - y_2 - z_2|)^{m + \beta}}
\end{align*}
and
\begin{align*}
\bigg( \iint_{\widehat{J}} \int_{\R^m} \bigg| \int_{J} K_{t_1,t_2}(x-y,z_1,z_2) dz_2 \bigg|^2 \Big(\frac{t_2}{t_2 + |y_2|}\Big)^{n \lambda_1} &\frac{dy_2dx_2 dt_2}{t_2^{m+1}} \bigg)^{\frac12} \\
&\lesssim |J|^{\frac12} \frac{t_1^{\alpha}}{(t_1 + |x_1 - y_1 - z_1|)^{n + \alpha}}.
\end{align*}
\item [(2)] Combinations of Carleson and H\"{o}lder conditions :
\begin{align*}
\bigg( \iint_{\widehat{I}} \int_{\Rn} \bigg| \int_{I} [ K_{t_1,t_2}(x-y,z_1,z_2) - K_{t_1,t_2}(x-y,z_1,&z_2') ] dz_1 \bigg|^2 \Big(\frac{t_1}{t_1 + |y_1|}\Big)^{n \lambda_1} \frac{dy_1dx_1 dt_1}{t_1^{n+1}}\bigg)^{\frac12} \\
&\lesssim |I|^{\frac12} \frac{|z_2 - z_2'|^{\beta}}{(t_2 + |x_2 - y_2 - z_2|)^{m + \beta}},
\end{align*}
\noindent whenever $|z_2 - z_2'| < t_2/2$. And
\begin{align*}
\bigg( \iint_{\widehat{J}} \int_{\R^m} \bigg| \int_{J} [ K_{t_1,t_2}(x-y,z_1,z_2) - K_{t_1,t_2}(x-y,&z_1',z_2) ] dz_2 \bigg|^2 \Big(\frac{t_2}{t_2 + |y_2|}\Big)^{m \lambda_2} \frac{dy_2dx_2 dt_2}{t_2^{m+1}} \bigg)^{\frac12} \\
&\lesssim |J|^{\frac12} \frac{|z_1 - z_1'|^{\alpha}}{(t_1 + |x_1 - y_1 - z_1|)^{n + \alpha}},
\end{align*}
whenever $|z_1 - z_1'| < t_1/2$.
\end{enumerate}
\noindent\textbf{Assumption 3 (Bi-parameter Carleson condition).}
Let $\mathcal{D}=\mathcal{D}_n \times \mathcal{D}_m$, where $\mathcal{D}_n$ is a dyadic grid in $\Rn$ and $\mathcal{D}_m$ is a dyadic grid in $\R^m$. For $I \in \mathcal{D}_n$, let $W_I = I \times (\ell(I)/2, \ell(I))$ be the associated Whitney region. Denote $n_1=n,n_2=m$ and
\begin{align*}
C_{I J}^{\mathcal{D}} &= \iint_{W_J} \iint_{W_I} \iint_{\R^{n+m}}|\theta_{t_1,t_2} \mathbf{1}(y_1, y_2)|^2
\bigg[\prod_{i=1}^2\Big(\frac{t_i}{t_i + |x_i - y_i|}\Big)^{n_i \lambda_i} \bigg] \frac{dy_1dy_2}{t_1^nt_2^m}
\frac{dx_1 dt_1}{t_1} \frac{dx_2 dt_2}{t_2} .
\end{align*}
We assume the following $bi$-$parameter \ Carleson \ condition$: For every $\mathcal{D}=\mathcal{D}_n \times \mathcal{D}_m$ there holds that
\begin{equation}\label{Car-condition}
\sum_{\substack{I \times J \in \mathcal{D} \\ I \times J \subset \Omega}} C_{I J}^{\mathcal{D}} \lesssim |\Omega|
\end{equation}
for all sets $\Omega \subset \R^{n+m}$ such that $|\Omega| < \infty$ and such that for every $x \in \Omega$ there exists $I \times J \in \mathcal{D}$ so that $x \in I \times J \subset \Omega$.

We now formulate the main result of this paper as follows.
\begin{theorem}
Let $\lambda_1,\lambda_2 >2$, $0 < \alpha \leq n(\lambda_1 -2)/2$ and $0 < \beta \leq m(\lambda_2 -2)/2$. Assume that the kernel $K_{t_1,t_2}$ satisfies the Assumptions 1-3. Then there holds that
$$\big\| g_{\lambda_1,\lambda_2}^*(f) \big\|_{L^2(\R^{n+m})} \lesssim \big\| f \big\|_{L^2(\R^{n+m})},$$
where the implied constant depends only on the assumptions.
\end{theorem}

\begin{remark}In section $\ref{Sec-necessity}$, we shall show that the bi-parameter Carleson condition is necessary for $g_{\lambda_1,\lambda_2}^*$-function to be bounded on $L^2(\R^{n+m})$. Moreover, Assumption 2 and Assumption 3 are much weaker than the similar conditions used in \cite{M2014}, since here two terms (both less than one) were added and more integrals related to $y_1$ or $y_2$ were used in our assumptions.
\end{remark}
\section{The Probabilistic Reduction}
In this section, our goal is to simplify the proof of the main result. First, we recall the definitions of random dyadic grids, good/bad cubes, Haar function on $\Rn$ which can be found in \cite{H}, \cite{M2012} and \cite{NTV}.
\subsection{Random Dyadic Grids}Let $\beta_n = \{ \beta_n^j\}_{j \in \Z}$, where
$\beta_n^j \in \{0,1\}^n$. Let $\mathcal{D}_n^0$ be the standard dyadic grids on $\Rn$. We define the new dyadic
grids in $\Rn$ by 
$$ \mathcal{D}_n = \Big\{I + \beta_n; I \in \mathcal{D}_n^0 \Big\} := \Big\{I + \sum_{j:2^{-j}<\ell(I)} 2^{-j} \beta_n^j; I \in \mathcal{D}_n^0 \Big\}.$$
Similarly, we can define the dyadic grids $\mathcal{D}_m$ in $\R^m$. There is a natural product probability structure on
$(\{0, 1\}^n)^{\Z}$ and $(\{0, 1\}^m)^{\Z}$. Therefore, we have independent random dyadic grids $\mathcal{D}_n$ and $\mathcal{D}_m$, respectively.

\subsection{Good and Bad Cubes.}
A cube $I \in \mathcal{D}_n$ is said to be $bad$ if there exists a $J \in \mathcal{D}_n$ with $\ell(J) \geq 2^r \ell(I)$ such that
$dist(I,\partial J) \leq \ell(I)^{\gamma_n} \ell(J)^{1-\gamma_n}$. Otherwise, $I$ is called $good$.
Here $r \in \Z_+$ and $\gamma_n \in (0,\frac12)$ are given parameters. 
Roughly speaking, a dyadic cube $I$ will be bad if it is relatively close to the boundary of a much bigger dyadic cube. Denote
$\pi_{good}^n = \mathbb{P}_{\beta_n}(I + \beta_n \ \text{is \ good}) = \mathbb{E}_{\beta_n}(\mathbf{1}_{good}(I+\beta_n))$.
Then $\pi_{good}^n$ is independent of $I \in \mathcal{D}_n^0$, and the parameter $r$ is a fixed constant so that $\pi_{good}^n,\pi_{good}^m > 0$.

Throughout this article, we take $\gamma_n = \frac{\alpha}{2(n+\alpha)}$, where $\alpha > 0$ appears in the kernel estimates. It is important to observe that the position and goodness of a cube $I \in \mathcal{D}_n^0$ are independent. Indeed, according to the definition, the spatial position of
$$ I + \beta_n = I + \sum_{j:2^{-j} < \ell(I)} 2^{-j} \beta_n^j $$ depends only on $\beta_n^j$ for $2^{-j} < \ell(I)$. On the other hand, the relative position of $I + \beta_n$ with respect to a bigger cube
$$ J + \beta_n = I + \sum_{j:2^{-j}<\ell(I)} 2^{-j} \beta_n^j + \sum_{j:\ell(I) \leq 2^{-j} < \ell(J)} 2^{-j} \beta_n^j$$
depends only on $\beta_n^j$ for $\ell(I) \leq 2^{-j} < \ell(J)$. Thus, the position and goodness of $I + \beta_n$ are independent.

\subsection{Haar functions}
In order to decompose a function $f \in L^2$, we introduce the definition of the Haar functions on $\Rn$.
Let $h_I$ be an $L^2$ normalized Haar function related to $I \in \mathcal{D}_n$, where $\mathcal{D}_n$ is a dyadic grid on $\Rn$. With this we mean that $h_I$, $I = I_1 \times \cdots \times I_n$, is one of the $2^n$ functions $h_I^\eta$, $\eta={\eta_1,\ldots,\eta_n} \in \{0, 1\}^n$, defined by
$$ h_I^\eta = h_{I_1}^{\eta_1} \otimes \cdots \otimes h_{I_n}^{\eta_n},$$
where $h_{I_i}^0 = |I_i|^{-1/2} \mathbf{1}_{I_i}$ and $h_{I_i}^1 = |I_i|^{-1/2}(\mathbf{1}_{I_{i,l}}-\mathbf{1}_{I_{i,r}})$ for every $i = 1, \ldots, n$. Here $I_{i,l}$ and $I_{i,r}$ are the left and right halves of the interval $I_i$ respectively. If $\eta \neq 0$, the Haar function is cancellative : $\int_\Rn h_I = 0$. All the cancellative Haar functions form an orthonormal basis of $L^2(\Rn)$. If $a \in L^2(\Rn)$, we may thus write
$$ a = \sum_{I \in \mathcal{D}_n} \sum_{\eta \in \{0, 1 \}^n \setminus \{0\}} \langle a, h_I^\eta \rangle h_I^\eta.$$
However, we suppress the finite $\eta$ summation and just write $a = \sum_I \langle a, h_I \rangle h_I.$ Using the corresponding product basis, we may expand a function $f$ defined in $\R^{n+m}$ in the following way:
$$ f = \sum_{I,J} f_{I J} h_{I \times J}
    := \sum_{I,J} \langle f, h_I \otimes h_J \rangle h_I \otimes h_J.$$
\subsection{Averaging over Good Whitney Regions}
Let $f \in L^2(\R^{n+m})$. Let $I_1, I_2 \in \mathcal{D}_n$ and $J_1, J_2 \in \mathcal{D}_m$.
Note that the position and goodness of $I + \beta_n$ are independent. Therefore, one can write,
\begin{align*}
&\big\| g_{\lambda_1,\lambda_2}^*(f) \big\|_{L^2(\R^{n+m})}^2 \\
&= \iint_{\R^{m+1}_{+}} \iint_{\R^{n+1}_{+}} \iint_{\R^{n+m}} |\theta_{t_1,t_2} f(x-y)|^2
\Big(\frac{t_1}{t_1 + |y_1|}\Big)^{n \lambda_1} \\
&\quad\quad \times \Big(\frac{t_2}{t_2 + |y_2|}\Big)^{m \lambda_2} \frac{dy_1}{t_1^n}
\frac{dy_2}{t_2^m} \frac{dx_1 dt_1}{t_1} \frac{dx_2 dt_2}{t_2} \\
&= \frac{1}{\pi_{good}^n} \frac{1}{\pi_{good}^m} \mathbb{E}_{\beta_n,\beta_m} \sum_{I_2,J_2: good} \iint_{W_{J_2}} \iint_{W_{I_2}} \iint_{\R^{n+m}} |\theta_{t_1,t_2} f(x-y)|^2 \\
&\quad\quad \times \Big(\frac{t_1}{t_1 + |y_1|}\Big)^{n \lambda_1} \Big(\frac{t_2}{t_2 + |y_2|}\Big)^{m \lambda_2} \frac{dy_1}{t_1^n} \frac{dy_2}{t_2^m} \frac{dx_1 dt_1}{t_1} \frac{dx_2 dt_2}{t_2} \\
&= \frac{1}{\pi_{good}^n} \frac{1}{\pi_{good}^m} \mathbb{E}_{\beta_n,\beta_m} \sum_{I_2,J_2: good} \iint_{W_{J_2}} \iint_{W_{I_2}} \iint_{\R^{n+m}} \Big| \sum_{I_1,J_1} f_{I_1 J_1} \theta_{t_1,t_2} h_{I_1 \times J_1}(x-y) \Big|^2 \\
&\quad\quad \times \Big(\frac{t_1}{t_1 + |y_1|}\Big)^{n \lambda_1} \Big(\frac{t_2}{t_2 + |y_2|}\Big)^{m \lambda_2} \frac{dy_1}{t_1^n} \frac{dy_2}{t_2^m} \frac{dx_1 dt_1}{t_1} \frac{dx_2 dt_2}{t_2}.
\end{align*}
Consequently, we are reduced to bound the sum
\begin{align*}
\mathcal{G}:= &\sum_{I_2,J_2: good} \iint_{W_{J_2}} \iint_{W_{I_2}} \iint_{\R^{n+m}} \Big| \sum_{I_1,J_1} f_{I_1 J_1} \theta_{t_1,t_2} h_{I_1 \times J_1}(x-y) \Big|^2 \\
&\quad\quad\quad  \times \Big(\frac{t_1}{t_1 + |y_1|}\Big)^{n \lambda_1} \Big(\frac{t_2}{t_2 + |y_2|}\Big)^{m \lambda_2} \frac{dy_1}{t_1^n} \frac{dy_2}{t_2^m} \frac{dx_1 dt_1}{t_1} \frac{dx_2 dt_2}{t_2}.
\end{align*}
Furthermore, we can carry out the decomposition
$$ \mathcal{G} \lesssim \mathcal{G}_{<,<} + \mathcal{G}_{<,\geq} + \mathcal{G}_{\geq,<} + \mathcal{G}_{\geq,\geq},$$
where
\begin{align*}
\mathcal{G}_{<,<}:= &\sum_{I_2,J_2: good} \iint_{W_{J_2}} \iint_{W_{I_2}} \iint_{\R^{n+m}} \Big| \sum_{\substack{I_1,J_1 \\ \ell(I_1) < \ell(I_2) \\ \ell(J_1) < \ell(J_2)}} f_{I_1 J_1} \theta_{t_1,t_2} h_{I_1 \times J_1}(x-y) \Big|^2 \\
&\quad\quad\quad  \times \Big(\frac{t_1}{t_1 + |y_1|}\Big)^{n \lambda_1} \Big(\frac{t_2}{t_2 + |y_2|}\Big)^{m \lambda_2} \frac{dy_1}{t_1^n} \frac{dy_2}{t_2^m} \frac{dx_1 dt_1}{t_1} \frac{dx_2 dt_2}{t_2},
\end{align*}
and the others are completely similar.

Sequentially, it is enough to focus on estimating the four pieces: $\mathcal{G}_{<,<}$, $\mathcal{G}_{<,\geq}$, $\mathcal{G}_{\geq,<}$, $\mathcal{G}_{\geq,\geq}$ in the following sections.
\section{The Case : $\ell(I_1) < \ell(I_2)$ and $\ell(J_1) < \ell(J_2)$}

For the sake of convenience, we first present two key lemmas, which will be used later.
\begin{lemma}[\cite{LL}, \cite{M2014}]\label{A-I1-I2}
Let $$A_{I_1 I_2}=\frac{\ell(I_1)^{\alpha/2} \ell(I_2)^{\alpha/2}}{D(I_1,I_2)^{n+\alpha}} |I_1|^{1/2} |I_2|^{1/2},$$
where the long distance $D(I_1,I_2)=\ell(I_1) + \ell(I_2) + d(I_1,I_2)$, $I_1,I_2 \in \mathcal{D}_n$ and $\alpha > 0$. Then, for any $x_{I_1}, y_{I_2} \geq 0$, we have the following estimate,
$$\Big( \sum_{I_1,I_2} A_{I_1 I_2} x_{I_1} y_{I_2} \Big)^2 \lesssim \sum_{I_1} x_{I_1}^2
\times \sum_{I_2} y_{I_2}^2.$$
In particular, there holds that:
$$ \sum_{I_2} \Big[ \sum_{I_1} A_{I_1 I_2} x_{I_1} \Big]^2 \lesssim \sum_{I_1}x_{I_1}^2. $$
\end{lemma}
\qed

\begin{lemma}\label{estimate-1}
Let $0 < \alpha \leq n(\lambda_1 -2)/2$. For a given cube $I_2 \in \mathcal{D}_n$ and $(x_1,t_1) \in W_{I_2}$,
the following inequality holds,
\begin{equation*}
\bigg[ \int_{\Rn} \bigg(\int_{I_1} \frac{dz_1}{(t_1 + |x_1 - y_1 - z_1|)^{n+\alpha}} \bigg)^2 \Big(\frac{t_1}{t_1 + |y_1|}\Big)^{n \lambda_1} \frac{dy_1}{t_1^n} \bigg]^{1/2}
\lesssim \frac{|I_1|}{(\ell(I_2) + d(I_1,I_2))^{n + \alpha}}.
\end{equation*}
\end{lemma}

\begin{proof}
Fixed $(x_1,t_1) \in W_{I_1}$. If $|y_1| \leq \frac{1}{2} d(I_1,I_2)$, then
$$ t_1 + |x_1 - y_1 - z_1| \gtrsim \ell(I_2) + |x_1 - z_1| - |y_1| \gtrsim \ell(I_2) + d(I_1,I_2).$$
Thus, it follows that
\begin{align*}
\bigg[ \int_{|y_1| \leq \frac{1}{2} d(I_1,I_2)} \bigg(\int_{I_1} \frac{dz_1}{(t_1 + |x_1 - y_1 - z_1|)^{n+\alpha}} \bigg)^2 \Big(\frac{t_1}{t_1 + |y_1|}\Big)^{n \lambda_1} \frac{dy_1}{t_1^n} \bigg]^{1/2}
\lesssim \frac{|I_1|}{(\ell(I_2) + d(I_1,I_2))^{n + \alpha}}.
\end{align*}

If $|y_1| > \frac{1}{2} d(I_1,I_2)$, then
$$ \Big(\frac{t_1}{t_1 + |y_1|}\Big)^{n \lambda_1} \frac{1}{t_1^n}
\lesssim \frac{\ell(I_2)^{n \lambda_1 - n}}{(\ell(I_2) + d(I_1,I_2))^{n \lambda_1}}.$$
Hence, we have
\begin{align*}
&\bigg[ \int_{|y_1| > \frac{1}{2} d(I_1,I_2)} \bigg(\int_{I_1} \frac{dz_1}{(t_1 + |x_1 - y_1 - z_1|)^{n+\alpha}} \bigg)^2 \Big(\frac{t_1}{t_1 + |y_1|}\Big)^{n \lambda_1} \frac{dy_1}{t_1^n} \bigg]^{1/2} \\
&\lesssim \frac{\ell(I_2)^{\frac{n \lambda_1}{2} - \frac{n}{2}}}{(\ell(I_2) + d(I_1,I_2))^{\frac{n \lambda_1}{2}}} \big\|\psi_{t_1}*\mathbf{1}_{I_1}\big\|_{L^2(\Rn)}
\lesssim \frac{\ell(I_2)^{\frac{n \lambda_1}{2} - \frac{n}{2}}}{(\ell(I_2) + d(I_1,I_2))^{\frac{n \lambda_1}{2}}}
\big\|\psi_{t_1}\big\|_{L^2(\Rn)} |I_1| \\
&\lesssim \frac{\ell(I_2)^{\frac{n \lambda_1}{2} - n -\alpha}}{(\ell(I_2) + d(I_1,I_2))^{\frac{n \lambda_1}{2}}} |I_1|
\lesssim \frac{|I_1|}{(\ell(I_2) + d(I_1,I_2))^{n + \alpha}}.
\end{align*}
where $\psi_{t_1}(z_1)=(t_1 + |z_1|)^{-n - \alpha}$ and we have used the condition $0 < \alpha \leq n(\lambda_1 -2)/2$ in the last step.
\end{proof}

Now we turn our attention to the estimate of $\mathcal{G}_{<,<}$. An easy consequence of the H\"{o}lder estimates of the kernel $K_{t_1,t_2}$ is that:
\begin{align*}
|\theta_{t_1,t_2} h_{I_1 \times J_1}(x-y)|
&\lesssim |I_1|^{-1/2} \int_{I_1} \frac{\ell(I_1)^\alpha}{(t_1 + |x_1 - y_1 - z_1|)^{n+\alpha}} dz_1\\
&\quad\quad\times |J_1|^{-1/2} \int_{J_1} \frac{\ell(J_1)^\beta}{(t_2 + |x_2 - y_2 - z_2|)^{m + \beta}} dz_2.
\end{align*}
Moreover, by Lemma $\ref{estimate-1}$, we can obtain that
\begin{equation}\aligned\label{iint-estimate}
&\mathcal{P}(x,t)\\
&:=\bigg( \iint_{\R^{n+m}} |\theta_{t_1,t_2} h_{I_1 \times J_1}(x-y)|^2
\Big(\frac{t_1}{t_1 + |y_1|}\Big)^{n \lambda_1} \Big(\frac{t_2}{t_2 + |y_2|}\Big)^{m \lambda_2} \frac{dy_1}{t_1^n}
\frac{dy_2}{t_2^m} \bigg)^{1/2} \\
&\lesssim |I_1|^{-1/2} \bigg[ \int_{\Rn} \bigg(\int_{I_1} \frac{\ell(I_1)^\alpha \ dz_1}{(t_1 + |x_1 - y_1 - z_1|)^{n+\alpha}} \bigg)^2 \Big(\frac{t_1}{t_1 + |y_1|}\Big)^{n \lambda_1} \frac{dy_1}{t_1^n} \bigg]^{1/2}  \\
&\quad \times |J_1|^{-1/2} \bigg[ \int_{\R^m} \bigg(\int_{J_1} \frac{\ell(J_1)^\beta \ dz_2}{(t_2 + |x_2 - y_2 - z_2|)^{m+\beta}} \bigg)^2 \Big(\frac{t_2}{t_2 + |y_2|}\Big)^{m \lambda_2} \frac{dy_2}{t_2^n} \bigg]^{1/2}  \\
&\lesssim \frac{\ell(I_1)^\alpha}{(\ell(I_2) + d(I_1,I_2))^{n + \alpha}} |I_1|^{1/2}
\frac{\ell(J_1)^\beta}{(\ell(J_2) + d(J_1,J_2))^{m + \beta}} |J_1|^{1/2} .
\endaligned
\end{equation}
Since $\ell(I_1) < \ell(I_2)$ and $\ell(J_1) < \ell(J_2)$, then we get
$$\mathcal{P}(x,t) \lesssim A_{I_1 I_2}|I_2|^{-1/2} \cdot A_{J_1 J_2}|J_2|^{-1/2}.$$
Therefore, by Minkowski's inequality and Lemma $\ref{A-I1-I2}$, it now follows that
\begin{align*}
\mathcal{G}_{<,<}
&\lesssim \sum_{I_2,J_2:good} \iint_{W_{J_2}} \iint_{W_{I_2}} \bigg[ \sum_{\substack{\ell(I_1) < \ell(I_2) \\
\ell(J_1) < \ell(J_2)}} |f_{I_1 J_1}| \bigg( \iint_{\R^{n+m}} |\theta_{t_1,t_2} h_{I_1 \times J_1}(x-y)|^2 \\
&\quad\quad \times \Big(\frac{t_1}{t_1 + |y_1|}\Big)^{n \lambda_1} \Big(\frac{t_2}{t_2 + |y_2|}\Big)^{m \lambda_2} \frac{dy_1}{t_1^n} \frac{dy_2}{t_2^m} \bigg)^{1/2} \bigg]^2 \frac{dx_1 dt_1}{t_1} \frac{dx_2 dt_2}{t_2} \\
&\lesssim \sum_{J_2} \sum_{I_2} \Big[ \sum_{I_1} A_{I_1 I_2} \sum_{J_1} A_{J_1 J_2} |f_{I_1 J_1}| \Big]^2 \\
&\lesssim \sum_{J_2} \sum_{I_1} \Big[ \sum_{J_1} A_{J_1 J_2} |f_{I_1 J_1}| \Big]^2 \\
&\lesssim \sum_{I_1} \sum_{J_1} |f_{I_1 J_1}|^2
=\big\| f \big\|_{L^2(\R^{n+m})}^2 .
\end{align*}
\section{The Case : $\ell(I_1) \geq \ell(I_2)$ and $\ell(J_1) < \ell(J_2)$}

In any case,we perform the splitting
$$ \sum_{\ell(I_1) \geq \ell(I_2)}
=\sum_{\substack{\ell(I_1) \geq \ell(I_2) \\ d(I_1,I_2) > \ell(I_2)^{\gamma_n} \ell(I_1)^{1-\gamma_n}}}
+\sum_{\substack{\ell(I_1) > 2^r \ell(I_2)\\d(I_1,I_2) \leq \ell(I_2)^{\gamma_n} \ell(I_1)^{1-\gamma_n}}}
+\sum_{\substack{\ell(I_2) \leq \ell(I_1) \leq 2^r \ell(I_2) \\ d(I_1,I_2) \leq \ell(I_2)^{\gamma_n}\ell(I_1)^{1-\gamma_n}}}. $$
These three parts are called separated, nested and adjacent respectively. The term nested makes sense, since the summing conditions that $I_2$ is good actually imply that $I_1$ is the ancestor of $I_2$. Thus, there holds
$$ \mathcal{G}_{\geq,<} \lesssim \mathcal{G}_{sep,<} + \mathcal{G}_{nes,<} + \mathcal{G}_{adj,<},$$
where
\begin{align*}
\mathcal{G}_{sep,<}
&= \sum_{I_2,J_2: good} \iint_{W_{J_2}} \iint_{W_{I_2}} \iint_{\R^{n+m}} \Big| \sum_{\substack{I_1:\ell(I_1) \geq \ell(I_2) \\ d(I_1,I_2) > \ell(I_2)^{\gamma_n} \ell(I_1)^{1-\gamma_n}}} \sum_{J_1:\ell(J_1) < \ell(J_2)} f_{I_1 J_1} \\
&\quad\quad \times \theta_{t_1,t_2} h_{I_1 \times J_1}(x-y) \Big|^2 \Big(\frac{t_1}{t_1 + |y_1|}\Big)^{n \lambda_1} \Big(\frac{t_2}{t_2 + |y_2|}\Big)^{m \lambda_2} \frac{dy_1}{t_1^n} \frac{dy_2}{t_2^m} \frac{dx_1 dt_1}{t_1}
\frac{dx_2 dt_2}{t_2},
\end{align*}
\begin{align*}
\mathcal{G}_{nes,<}
&= \sum_{I_2,J_2: good} \iint_{W_{J_2}} \iint_{W_{I_2}} \iint_{\R^{n+m}} \Big| \sum_{\substack{I_1:\ell(I_1) > 2^r \ell(I_2)\\d(I_1,I_2) \leq \ell(I_2)^{\gamma_n} \ell(I_1)^{1-\gamma_n}}} \sum_{J_1:\ell(J_1)<\ell(J_2)}f_{I_1 J_1}  \\
&\quad\quad  \times \theta_{t_1,t_2} h_{I_1 \times J_1}(x-y) \Big|^2 \Big(\frac{t_1}{t_1 + |y_1|}\Big)^{n \lambda_1} \Big(\frac{t_2}{t_2 + |y_2|}\Big)^{m \lambda_2} \frac{dy_1}{t_1^n} \frac{dy_2}{t_2^m} \frac{dx_1 dt_1}{t_1} \frac{dx_2 dt_2}{t_2},
\end{align*}
and
\begin{align*}
\mathcal{G}_{adj,<}
&= \sum_{I_2,J_2:good} \iint_{W_{J_2}} \iint_{W_{I_2}} \iint_{\R^{n+m}} \Big| \sum_{\substack{I_1:\ell(I_2) \leq \ell(I_1) \leq 2^r \ell(I_2) \\ d(I_1,I_2) \leq \ell(I_2)^{\gamma_n} \ell(I_1)^{1-\gamma_n}}} \sum_{J_1:\ell(J_1)<\ell(J_2)} f_{I_1 J_1} \\
&\quad\quad  \times \theta_{t_1,t_2} h_{I_1 \times J_1}(x-y) \Big|^2  \Big(\frac{t_1}{t_1 + |y_1|}\Big)^{n \lambda_1} \Big(\frac{t_2}{t_2 + |y_2|}\Big)^{m \lambda_2} \frac{dy_1}{t_1^n} \frac{dy_2}{t_2^m} \frac{dx_1 dt_1}{t_1} \frac{dx_2 dt_2}{t_2}.
\end{align*}

Now, we are in the position to estimate the above three terms, respectively.
\subsection{Separated Part $\mathcal{G}_{sep,<}$.}In this case, we note that the following inequality holds,
\begin{equation}\label{t1-alpha}
\frac{t_1^\alpha}{(\ell(I_2) + d(I_1,I_2))^{n + \alpha}} |I_1|^{1/2}
\lesssim A_{I_1 I_2}|I_2|^{-1/2}.
\end{equation}
Indeed, if $d(I_1,I_2) \geq \ell(I_1)$, then $D(I_1,I_2) \thicksim d(I_1,I_2)$. Therefore, we get
\begin{equation*}
\frac{t_1^\alpha}{(\ell(I_2) + d(I_1,I_2))^{n + \alpha}} |I_1|^{1/2}
\lesssim A_{I_1 I_2}|I_2|^{-1/2}.
\end{equation*}
If $d(I_1,I_2) < \ell(I_1)$, then $D(I_1,I_2) \thicksim \ell(I_1)$. Moreover, notice that $\gamma_n (n + \alpha) = \alpha/2$ and $d(I_1,I_2) > \ell(I_2)^{\gamma_n} \ell(I_1)^{1-\gamma_n}$, one may conclude that
\begin{equation*}
\frac{t_1^\alpha}{(\ell(I_2) + d(I_1,I_2))^{n + \alpha}} |I_1|^{1/2}
\lesssim \frac{\ell(I_1)^{\alpha/2} \ell(I_2)^{\alpha/2}}{\ell(I_1)^{n + \alpha}}|I_1|^{1/2}
\lesssim A_{I_1 I_2}|I_2|^{-1/2}.
\end{equation*}

It is obvious that the mixed H\"{o}lder and size condition implies that
\begin{align*}
|\theta_{t_1,t_2} h_{I_1 \times J_1}(x-y)|
&\lesssim |I_1|^{-1/2} \int_{I_1} \frac{t_1^\alpha}{(t_1 + |x_1 - y_1 - z_1|)^{n+\alpha}} dz_1\\
&\quad \times |J_1|^{-1/2} \int_{J_1} \frac{\ell(J_1)^\beta}{(t_2 + |x_2 - y_2 - z_2|)^{m + \beta}} dz_2.
\end{align*}
Thus, combining Lemma $\ref{estimate-1}$ with $(\ref{t1-alpha})$, one can obtain
\begin{equation}\aligned\label{Rnm-theta}
\mathcal{P}(x,t)
&\lesssim |I_1|^{-1/2} \bigg[ \int_{\Rn} \bigg(\int_{I_1} \frac{t_1^\alpha \ dz_1}{(t_1 + |x_1 - y_1 - z_1|)^{n+\alpha}} \bigg)^2 \Big(\frac{t_1}{t_1 + |y_1|}\Big)^{n \lambda_1} \frac{dy_1}{t_1^n} \bigg]^{1/2}  \\
&\quad \times |J_1|^{-1/2} \bigg[ \int_{\R^m} \bigg(\int_{J_1} \frac{\ell(J_1)^\beta \ dz_2}{(t_2 + |x_2 - y_2 - z_2|)^{m+\beta}} \bigg)^2 \Big(\frac{t_2}{t_2 + |y_2|}\Big)^{m \lambda_2} \frac{dy_2}{t_2^n} \bigg]^{1/2}  \\
&\lesssim \frac{t_1^\alpha}{(\ell(I_2) + d(I_1,I_2))^{n + \alpha}} |I_1|^{1/2}
\frac{\ell(J_1)^\beta}{(\ell(J_2) + d(J_1,J_2))^{m + \beta}} |J_1|^{1/2} \\
&\lesssim \frac{\ell(I_2)^\alpha}{d(I_1,I_2)^{n + \alpha}} |I_1|^{1/2} A_{J_1 J_2} |J_2|^{-1/2}
\lesssim A_{I_1 I_2}|I_2|^{-1/2} \cdot A_{J_1 J_2}|J_2|^{-1/2}.
\endaligned
\end{equation}
Consequently, by the similar argument as $\mathcal{G}_{<,<}$, we have
$$\mathcal{G}_{sep,<} \lesssim \big\| f \big\|_{L^2(\R^{n+m})}^2.$$
\subsection{Adjacent Part $\mathcal{G}_{adj,<}$.}
The summation conditions $\ell(I_2) \leq \ell(I_1) \leq 2^r \ell(I_2)$ and $d(I_1,I_2) \leq \ell(I_2)^{\gamma_n} \ell(I_1)^{1- \gamma_n}$ indicate that $\ell(I_1) \thicksim \ell(I_2) \thicksim D(I_1,I_2)$. Thus,
$$ \frac{\ell(I_2)^\alpha}{(\ell(I_2) + d(I_1,I_2))^{n + \alpha}} |I_1|^{1/2}
\lesssim \ell(I_2)^{-n}
\thicksim \frac{\ell(I_1)^{\alpha/2}\ell(I_2)^{\alpha/2}}{D(I_1,I_2)^{n+\alpha}}.$$
It follows from $(\ref{Rnm-theta})$ that
\begin{align*}
\mathcal{P}(x,t)
\lesssim \frac{\ell(I_2)^\alpha}{(\ell(I_2) + d(I_1,I_2))^{n + \alpha}} |I_1|^{1/2} \cdot A_{J_1 J_2} |J_2|^{-1/2}
\lesssim  A_{I_1 I_2}|I_2|^{-1/2} \cdot A_{J_1 J_2} |J_2|^{-1/2}.
\end{align*}
Therefore, exactly as we have seen before,
$$ \mathcal{G}_{adj,<} \lesssim \big\| f \big\|_{L^2(\R^{n+m})}^2. $$
\subsection{Nested Part $\mathcal{G}_{nes,<}$.}
We use $I^{(k)} \in \mathcal{D}_n$ to denote the unique cube for which $\ell(I^{(k)}) = 2^k \ell(I)$ and $I \subset I^{(k)}$. We call $I^{(k)}$ as the $k$ generation older dyadic ancestor of $I$. In this case, by the goodness of $I_2$, it must actually have $I_2 \subsetneq I_1$. That is, $I_1$ is the ancestor of $I_2$.
This enables us to write
\begin{align*}
\mathcal{G}_{nes,<}
&= \sum_{I,J_2:good} \iint_{W_{J_2}} \iint_{W_{I}} \iint_{\R^{n+m}} \Big| \sum_{k=1}^\infty
\sum_{J_1:\ell(J_1)<\ell(J_2)} f_{I^{(k)} J_1} \theta_{t_1,t_2} h_{I^{(k)} \times J_1}(x-y) \Big|^2 \\
&\quad\quad\quad  \times   \Big(\frac{t_1}{t_1 + |y_1|}\Big)^{n \lambda_1} \Big(\frac{t_2}{t_2 + |y_2|}\Big)^{m \lambda_2} \frac{dy_1}{t_1^n} \frac{dy_2}{t_2^m} \frac{dx_1 dt_1}{t_1} \frac{dx_2 dt_2}{t_2}.
\end{align*}

We introduce the notation
$$ s_I^k = - \mathbf{1}_{(I^{(k-1)})^c} \langle h_{I^{(k)}} \rangle_{I^{(k-1)}} + \sum_{\substack{I' \in ch(I^{(k)}) \\ I' \neq I^{(k-1)}}} \mathbf{1}_{I'} h_{I^{(k)}}.$$
Then, it is easy to check that
\begin{equation}\label{h-I-k}
 h_{I^{(k)}} = s_I^k + \langle h_{I^{(k)}} \rangle_{I^{(k-1)}},
\end{equation}
$\supp s_I^k \subset (I^{(k-1)})^c$, and $|s_I^k| \lesssim |I^{(k)}|^{-1/2}$.

Denote $f_{J_1} = \langle f,h_{J_1} \rangle$ so that $f_{J_1} = \int_{\R^m} f(y_1,y_2) dy_2$, $y_1 \in \Rn$. 
Then, we split
$$ \mathcal{G}_{nes,<} \lesssim \mathcal{G}_{mod,<} + \mathcal{G}_{Car,<} \ \ ,$$
where
\begin{align*}
\mathcal{G}_{mod,<}
&= \sum_{I,J_2: good} \iint_{W_{J_2}} \iint_{W_{I}} \iint_{\R^{n+m}} \Big| \sum_{k=1}^\infty \sum_{J_1:\ell(J_1)<\ell(J_2)}
f_{I^{(k)} J_1} \theta_{t_1,t_2}(s_I^k \otimes h_{J_1})(x-y) \Big|^2 \\
&\quad\quad\quad\quad\quad\quad  \times \Big(\frac{t_1}{t_1 + |y_1|}\Big)^{n \lambda_1} \Big(\frac{t_2}{t_2 + |y_2|}\Big)^{m \lambda_2} \frac{dy_1}{t_1^n} \frac{dy_2}{t_2^m} \frac{dx_1 dt_1}{t_1} \frac{dx_2 dt_2}{t_2}
\end{align*}
and
\begin{align*}
\mathcal{G}_{Car,<}
&= \sum_{I,J_2: good} \iint_{W_{J_2}} \iint_{W_{I}} \iint_{\R^{n+m}} \Big| \sum_{J_1:\ell(J_1)<\ell(J_2)}
f_{I^{(k)} J_1} \theta_{t_1,t_2}(\mathbf{1} \otimes h_{J_1})(x-y)  \\
&\quad \times \sum_{k=1}^\infty \langle \Delta_{I^{(k)}} f_{J_1} \rangle_{I^{(k-1)}}\Big|^2 \Big(\frac{t_1}{t_1 + |y_1|}\Big)^{n \lambda_1} \Big(\frac{t_2}{t_2 + |y_2|}\Big)^{m \lambda_2} \frac{dy_1}{t_1^n} \frac{dy_2}{t_2^m} \frac{dx_1 dt_1}{t_1} \frac{dx_2 dt_2}{t_2}\\
&= \sum_{I,J_2: good} \iint_{W_{J_2}} \iint_{W_{I}} \iint_{\R^{n+m}} \Big| \sum_{J_1:\ell(J_1)<\ell(J_2)}
\langle f_{J_1} \rangle_{I} \theta_{t_1,t_2}(\mathbf{1} \otimes h_{J_1})(x-y) \Big|^2  \\
&\quad \times \Big(\frac{t_1}{t_1 + |y_1|}\Big)^{n \lambda_1} \Big(\frac{t_2}{t_2 + |y_2|}\Big)^{m \lambda_2} \frac{dy_1}{t_1^n} \frac{dy_2}{t_2^m} \frac{dx_1 dt_1}{t_1} \frac{dx_2 dt_2}{t_2}.
\end{align*}

\vspace{0.3cm}
\noindent\textbf{$\bullet$ Estimate of $\mathcal{G}_{mod,<}$.} We need the following lemma. 

\begin{lemma}\label{2-alpha-k-2}
Let $0 < \beta \leq m(\lambda_2 -2)/2$ and $k \in \N_+$. Given cubes $I \in \mathcal{D}_n$, $J_1,J_2 \in \mathcal{D}_m$, $(x_1,t_1) \in W_I$, and $(x_2,t_2) \in W_{J_1}$, the following estimate holds
\begin{align*}
\mathcal{Q}(x,t)&:=\bigg( \iint_{\R^{n+m}} |\theta_{t_1,t_2}(s_I^k \otimes h_{J_1})(x-y)|^2
\Big(\frac{t_1}{t_1 + |y_1|}\Big)^{n \lambda_1} \Big(\frac{t_2}{t_2 + |y_2|}\Big)^{m \lambda_2} \frac{dy_1}{t_1^n}
\frac{dy_2}{t_2^m} \bigg)^{1/2} \\
&\lesssim 2^{-\alpha k/2} |I^{(k)}|^{-1/2}|J_1|^{1/2} \frac{\ell(J_1)^\beta}{(\ell(J_2) + d(J_1,J_2))^{m + \beta}}
\end{align*}
\end{lemma}

\begin{proof}
By using the mixed H\"{o}lder and size condition, it yields that
\begin{align*}
|\theta_{t_1,t_2}(s_I^k \otimes h_{J_1})(x-y)|
&\lesssim |I^{(k)}|^{-1/2} \int_{(I^{(k-1)})^c}\frac{t_1^\alpha }{(t_1 + |x_1 - y_1 - z_1|)^{n+\alpha}} dz_1 \\
&\quad \times |J_1|^{-1/2} \int_{J_1} \frac{\ell(J_1)^\beta}{(t_2 + |x_2 - y_2 - z_2|)^{m + \beta}} dz_2
\end{align*}
Similarly as in $(\ref{iint-estimate})$, we only need to show
\begin{equation}\aligned\label{estimate-2}
\mathcal{K}:= \bigg[ \int_{\Rn} \bigg(\int_{(I^{(k-1)})^c} \frac{t_1^\alpha \ dz_1}{(t_1 + |x_1 - y_1 - z_1|)^{n + \alpha}} \bigg)^2 \Big(\frac{t_1}{t_1 + |y_1|}\Big)^{n \lambda_1} \frac{dy_1}{t_1^n} \bigg]^{1/2}
\thicksim 2^{-\alpha k/2}.
\endaligned
\end{equation}
Indeed, if $k \leq r$,
\begin{align*}
\mathcal{K} & \lesssim \bigg[ \int_{\Rn} \bigg(\int_{\Rn} \frac{\ell(I)^\alpha \ dz_1}{(\ell(I) + |x_1 - y_1 - z_1|)^{n + \alpha}} \bigg)^2 \Big(\frac{t_1}{t_1 + |y_1|}\Big)^{n \lambda_1} \frac{dy_1}{t_1^n} \bigg]^{1/2} \\
&\lesssim \int_{\Rn} \frac{\ell(I)^\alpha}{(\ell(I) + |x_1 - z_1|)^{n + \alpha}} dz_1 \\
&\lesssim \ell(I)^{-n} |I| + \ell(I)^{\alpha} \int_{I^c} \frac{dz_1}{|z_1 - x_1|^{n + \alpha}}
\lesssim 1 \thicksim 2^{-\alpha k/2}.
\end{align*}
If $k > r$, the goodness of $I$ gives that
$$ d(I,(I^{(k-1)})^c) > \ell(I)^{\gamma_n} \ell(I^{(k-1)})^{1-\gamma_n}= 2^{(k-1)(1-\gamma_n)} \ell(I) \gtrsim 2^{k/2} \ell(I).$$
Therefore, we obtain
\begin{align*}
\int_{(I^{(k-1)})^c} \frac{\ell(I)^\alpha}{|z_1 - x_1|^{n+\alpha}} dz_1
&\leq \int_{B(x_1,d(I,(I^{(k-1)})^c))} \frac{\ell(I)^\alpha}{|z_1 - x_1|^{n+\alpha}} dz_1 \\
&\lesssim \ell(I)^\alpha d(I,(I^{(k-1)})^c)^{-\alpha} \lesssim 2^{-\alpha k/2}.
\end{align*}
Given $y_1 \in \Rn$, we introduce the notation
$$ E_1 = \big\{z_1 \in (I^{(k-1)})^c; |z_1 - x_1| \geq 2|y_1| \big\}, \
   E_2 = \big\{z_1 \in (I^{(k-1)})^c; |z_1 - x_1| < 2|y_1| \big\}.$$
Then, it follows that
\begin{align*}
\mathcal{K} & \lesssim \bigg[ \int_{\Rn} \bigg(\int_{E_1} \frac{\ell(I)^\alpha \ dz_1}{(\ell(I) + |x_1 - y_1 - z_1|)^{n + \alpha}} \bigg)^2 \Big(\frac{t_1}{t_1 + |y_1|}\Big)^{n \lambda_1} \frac{dy_1}{t_1^n} \bigg]^{1/2} \\
&\quad + \bigg[ \int_{\Rn} \bigg(\int_{E_2} \frac{\ell(I)^\alpha \ dz_1}{(\ell(I) + |x_1 - y_1 - z_1|)^{n + \alpha}} \bigg)^2 \Big(\frac{t_1}{t_1 + |y_1|}\Big)^{n \lambda_1} \frac{dy_1}{t_1^n} \bigg]^{1/2} \\
&:= \mathcal{K}_1 + \mathcal{K}_2.
\end{align*}
Note that $ \ell(I) + |x_1 - y_1 - z_1| > |z_1 - x_1| - |y_1| \geq \frac12 |z_1 - x_1|$ whenever $z_1 \in E_1$. This yields that
$$ \mathcal{K}_1
\lesssim \int_{(I^{(k-1)})^c} \frac{\ell(I)^\alpha}{|z_1 - x_1|^{n+\alpha}}dz_1 \cdot
\bigg[ \int_{\Rn} \Big(\frac{t_1}{t_1 + |y_1|}\Big)^{n \lambda_1} \frac{dy_1}{t_1^n} \bigg]^{1/2}
\lesssim 2^{-\alpha k/2}.$$
As for $\mathcal{K}_2 $, let $\xi(z_1) = \frac{1}{t_1 + |z_1|^{n + \alpha}}, \ \ \eta(z_1) = \frac{\mathbf{1}_{E_2}(z_1)}{t_1 + |z_1-x_1|^{n + \alpha}}.$ By Young's inequality, we have
\begin{align*}
\mathcal{K}_2
&\lesssim \ell(I)^\alpha \bigg[ \int_{\Rn} \bigg(\int_{E_2} \frac{\ell(I)^{\frac{n \lambda_1}{2}-\frac{n}{2}}}{(\ell(I) + |x_1 - z_1|)^{\frac{n \lambda_1}{2}}} \frac{1}{(t_1 + |y_1 - z_1|)^{n + \alpha}} dz_1 \bigg)^2 dy_1 \bigg]^{1/2} \\
&\leq \ell(I)^\alpha \bigg[ \int_{\Rn} \bigg(\int_{E_2} \frac{\ell(I)^{\frac{n}{2}+\alpha}}{(\ell(I) + |x_1 - z_1|)^{n+\alpha}} \frac{1}{(t_1 + |y_1 - z_1|)^{n + \alpha}} dz_1 \bigg)^2 dy_1 \bigg]^{1/2} \\
&= \ell(I)^{\frac{n}{2} + 2 \alpha} \big\| \xi * \eta \big\|_{L^2(\Rn)}
\leq \ell(I)^{\frac{n}{2} + 2 \alpha} \big\| \xi \big\|_{L^2(\Rn)} \big\| \eta \big\|_{L^1(\Rn)} \\
&\lesssim \ell(I)^\alpha \int_{(I^{(k-1)})^c} \frac{\ell(I)^\alpha}{|z_1 - x_1|^{n+\alpha}}dz_1
\lesssim 2^{-\alpha k/2},
\end{align*}
\end{proof}

Thus, by Minkowski's inequality and Lemma $\ref{2-alpha-k-2}$, $\mathcal{G}_{mod,<}$ can be controlled by
\begin{align*}
&\mathcal{G}_{mod,<}\\&\lesssim \sum_{I,J_2:good} \iint_{W_{J_2}} \iint_{W_I} \bigg[\sum_{k=1}^\infty \sum_{J_1:\ell(J_1) < \ell(J_2)}
|f_{I^{(k)} J_1}| \bigg( \iint_{\R^{n+m}} |\theta_{t_1,t_2}(s_I^k \otimes h_{J_1})(x-y)|^2 \\
&\quad\quad\quad \times \Big(\frac{t_1}{t_1 + |y_1|}\Big)^{n \lambda_1} \Big(\frac{t_2}{t_2 + |y_2|}\Big)^{m \lambda_2} \frac{dy_1}{t_1^n} \frac{dy_2}{t_2^m} \bigg)^{1/2} \bigg]^2 \frac{dx_1 dt_1}{t_1} \frac{dx_2 dt_2}{t_2} \\
&\leq \sum_{I}\sum_{J_2} \iint_{W_{J_2}} \iint_{W_I} \bigg[\sum_{k=1}^\infty \frac1{2^{\alpha k/2}} \sum_{J_1:\ell(J_1) < \ell(J_2)}
A_{J_1 J_2} |f_{I^{(k)} J_1}| {(\frac {|J_1|}{|I^{(k)}|})^{1/2}} \bigg]^2 \frac{dx_1 dt_1dx_2 dt_2}{t_1t_2} \\
&\lesssim \sum_{I} \sum_{J_2} \bigg[ \sum_{k=1}^\infty 2^{-\alpha k/2} \Big( \frac{|I|}{|I^{(k)}|} \sum_{J_1:\ell(J_1) < \ell(J_2)} A_{J_1 J_2} |f_{I^{(k)} J_1}| \Big)^{1/2} \bigg]^2 \\
&\leq \bigg[ \sum_{k=1}^\infty 2^{-\alpha k/4} \cdot 2^{-\alpha k/4}\bigg( \sum_I \frac{|I|}{|I^{(k)}|} \sum_{J_2} \Big( \sum_{J_1:\ell(J_1) < \ell(J_2)} A_{J_1 J_2} |f_{I^{(k)} J_1}| \Big)^2 \bigg)^{1/2} \bigg]^2 \\
&\lesssim \sum_{k=1}^\infty 2^{-\alpha k/2} \sum_I \frac{|I|}{|I^{(k)}|} \sum_{J_2} \Big( \sum_{J_1:\ell(J_1) < \ell(J_2)} A_{J_1 J_2} |f_{I^{(k)} J_1}| \Big)^2 \\
&\lesssim \sum_{k=1}^\infty 2^{-\alpha k/2} \sum_{Q,J_1} \frac{|f_{Q J_1}|^2}{|Q|} \sum_{I:I^{(k)=Q}}|I|
\lesssim \big\| f \big\|_{L^2(\R^{n+m})}^2 .
\end{align*}
\qed

\noindent\textbf{$\bullet$ Estimate of $\mathcal{G}_{Car,<}$.} We need to use the following lemma.

\begin{lemma}\label{Car-1}
Let $J_1$, $J_2 \in \mathcal{D}_m$ be cubes, and $(x_2,t_2) \in W_{J_2}$. Then the Carleson condition holds
\begin{align*}
\mathcal{R}(x_2,t_2)&:= \sum_{I' \subset I} \iint_{W_{I'}} \iint_{\R^{n+m}} |\theta_{t_1,t_2}(\mathbf{1} \otimes h_{J_1})(x-y)|^2 \prod_{i=1}^2 \Big(\frac{t_i}{t_i + |y_i|}\Big)^{n \lambda_i}
\frac{dy_1}{t_1^n} \frac{dy_2}{t_2^m} \frac{dx_1 dt_1}{t_1} \\
&\lesssim |I| \bigg( \frac{\ell(J_1)^\beta \ |J_1|^{1/2}}{(\ell(J_2) + d(J_1,J_2))^{m + \beta}}\bigg)^2.
\end{align*}
\end{lemma}

\begin{proof}
The first step is to split $$\mathcal{R}(x_2,t_2) \lesssim \mathcal{R}_1(x_2,t_2) + \mathcal{R}_2(x_2,t_2),$$
where
\begin{align*}
\mathcal{R}_1(x_2,t_2) &= \iint_{\widehat{3I}} \iint_{\R^{n+m}} |\theta_{t_1,t_2}(\mathbf{1}_{3I} \otimes h_{J_1})(x-y)|^2\prod_{i=1}^2 \Big(\frac{t_i}{t_i + |y_i|}\Big)^{n \lambda_i}\frac{dy_1 dy_2}{t_1^n t_2^m} \frac{dx_1 dt_1}{t_1},
\end{align*}
and
\begin{align*}
\mathcal{R}_2(x_2,t_2) &= \iint_{\widehat{I}} \iint_{\R^{n+m}} |\theta_{t_1,t_2}(\mathbf{1}_{(3I)^c} \otimes h_{J_1})(x-y)|^2 \prod_{i=1}^2 \Big(\frac{t_i}{t_i + |y_i|}\Big)^{n \lambda_i} \frac{dy_1 dy_2}{t_1^n t_2^m} \frac{dx_1 dt_1}{t_1} \\
&:= \iint_{\widehat{I}} H(x,t) \frac{dx_1 dt_1}{t_1}.
\end{align*}

By the combinations of Carleson and H\"{o}lder conditions and Lemma $\ref{estimate-1}$, it follows that
\begin{align*}
&\mathcal{R}_1(x_2,t_2) \\&\lesssim  \frac {1}{|J_1|} \int_{\R^m} \iint_{\widehat{3I}} \int_{\Rn} \bigg| \int_{J_1} \int_{3I} [K_{t_1,t_2}(x-y,(z_1,z_2)) - K_{t_1,t_2}(x-y,(z_1,z_2 + c_{J_1}))] dz_1 dz_2 \bigg|^2 \\
&\quad\quad \times \Big(\frac{t_1}{t_1 + |y_1|}\Big)^{n \lambda_1} \frac{dy_1}{t_1^n} \frac{dx_1 dt_1}{t_1} \Big(\frac{t_2}{t_2 + |y_2|}\Big)^{m \lambda_2} \frac{dy_2}{t_2^m}  \\
& \frac {1}{|J_1|} \int_{\R^m} \bigg[ \int_{J_1} \bigg( \iint_{\widehat{3I}} \int_{\Rn} \bigg|  \int_{3I} [K_{t_1,t_2}(x-y,(z_1,z_2)) - K_{t_1,t_2}(x-y,(z_1,z_2+ c_{J_1}))] dz_1 \bigg|^2 \\
&\quad\quad \times \Big(\frac{t_1}{t_1 + |y_1|}\Big)^{n \lambda_1} \frac{dy_1}{t_1^n} \frac{dx_1 dt_1}{t_1} \bigg)^{1/2} dz_2 \bigg]^2 \Big(\frac{t_2}{t_2 + |y_2|}\Big)^{m \lambda_2} \frac{dy_2}{t_2^m}  \\
&\lesssim |I| |J_1|^{-1} \int_{\R^m} \bigg( \int_{J_1} \frac{\ell(J_1)^\beta dz_2}{(t_2 + |x_2 - y_2 - z_2|)^{m + \beta}} \bigg)^2 \Big(\frac{t_2}{t_2 + |y_2|}\Big)^{m \lambda_2} \frac{dy_2}{t_2^m}  \\
&\lesssim |I| \bigg( \frac{\ell(J_1)^\beta \ |J_1|^{1/2}}{(\ell(J_2) + d(J_1,J_2))^{m + \beta}}\bigg)^2 .
\end{align*}

The mixed H\"{o}lder and size estimate gives that
\begin{align*}
|\theta_{t_1,t_2}(\mathbf{1}_{(3I)^c} \otimes h_{J_1})(x-y)|
&\lesssim |J_1|^{-1/2} \int_{(3I)^c} \frac{t_1^\alpha}{(t_1 + |x_1 - y_1 - z_1|)^{n + \alpha}} dz_1 \\
&\quad \times \int_{J_1} \frac{\ell(J_1)^\beta}{(t_2 + |x_2 - y_2 - z_2|)^{m + \beta}} dz_2.
\end{align*}
Thus, by the estimates in Lemma $\ref{estimate-1}$ and $(\ref{estimate-2})$, one can deduce that
\begin{align*}
H(x,t)
&\lesssim |J_1|^{-1} \int_{\Rn} \bigg( \int_{(3I)^c} \frac{t_1^\alpha dz_1}{(t_1 + |x_1 - y_1 - z_1|)^{n+\alpha}} \bigg)^2 \Big(\frac{t_1}{t_1 + |y_1|}\Big)^{n \lambda_1} \frac{dy_1}{t_1^n} \\
&\quad\quad\quad \times \int_{\R^m} \bigg(\int_{J_1} \frac{\ell(J_1)^\beta dz_2}{(t_2 + |x_2 - y_2 - z_2|)^{m+\beta}} \bigg)^2 \Big(\frac{t_2}{t_2 + |y_2|}\Big)^{m \lambda_2} \frac{dy_2}{t_2^n}  \\
&\lesssim t_1^{2 \alpha} \ell(I)^{-2 \alpha} \bigg( \frac{\ell(J_1)^\beta \ |J_1|^{1/2}}{(\ell(J_2) + d(J_1,J_2))^{m + \beta}}\bigg)^2.
\end{align*}
Therefore, we obtain
\begin{align*}
&\mathcal{R}_2(x_2,t_2) = \iint_{\widehat{I}} H(x,t) \frac{dx_1 dt_1}{t_1} \\
&\lesssim |I| \ell(I)^{-2 \alpha} \int_{0}^{\ell(I)} t_1^{2 \alpha - 1} dt_1 \cdot
\bigg( \frac{\ell(J_1)^\beta \ |J_1|^{1/2}}{(\ell(J_2) + d(J_1,J_2))^{m + \beta}}\bigg)^2 \\
&\lesssim |I| \bigg( \frac{\ell(J_1)^\beta \ |J_1|^{1/2}}{(\ell(J_2) + d(J_1,J_2))^{m + \beta}}\bigg)^2.
\end{align*}
Thus, we finish the proof of Lemma $\ref{Car-1}$.

\end{proof}

Now we give the estimate for $\mathcal{G}_{Car,<}$. If $\ell(J_1) < \ell(J_2)$, then we have
$$\mathcal{R}(x_2,t_2) \lesssim |I|(A_{J_1,J_2} |J_2|^{-1/2})^2.$$
Therefore, we obtain the following estimate
\begin{align*}
\mathcal{G}_{Car,<}
&= \sum_{J_2: good} \iint_{W_{J_2}} \sum_{I:good} \iint_{W_{I}} \iint_{\R^{n+m}} \Big| \sum_{J_1:\ell(J_1)<\ell(J_2)}
\langle f_{J_1} \rangle_{I} \theta_{t_1,t_2}(\mathbf{1} \otimes h_{J_1})(x-y) \Big|^2  \\
&\quad\quad \times \Big(\frac{t_1}{t_1 + |y_1|}\Big)^{n \lambda_1} \Big(\frac{t_2}{t_2 + |y_2|}\Big)^{m \lambda_2} \frac{dy_1}{t_1^n} \frac{dy_2}{t_2^m} \frac{dx_1 dt_1}{t_1} \frac{dx_2 dt_2}{t_2} \\
&\leq \sum_{J_2} \iint_{W_{J_2}} \sum_I \bigg[ \sum_{J_1:\ell(J_1)<\ell(J_2)} \bigg(\iint_{W_{I}} \iint_{\R^{n+m}} \Big|
\langle f_{J_1} \rangle_{I} \theta_{t_1,t_2}(\mathbf{1} \otimes h_{J_1})(x-y) \Big|^2  \\
&\quad\quad \times \Big(\frac{t_1}{t_1 + |y_1|}\Big)^{n \lambda_1} \Big(\frac{t_2}{t_2 + |y_2|}\Big)^{m \lambda_2} \frac{dy_1}{t_1^n} \frac{dy_2}{t_2^m} \frac{dx_1 dt_1}{t_1} \bigg)^{1/2} \bigg]^2 \frac{dx_2 dt_2}{t_2} \\
&\leq \sum_{J_2} \iint_{W_{J_2}} \bigg[ \sum_{J_1:\ell(J_1)<\ell(J_2)} \bigg( \sum_I |\langle f_{J_1} \rangle_{I}|^2 \iint_{W_{I}} \iint_{\R^{n+m}} \Big| \theta_{t_1,t_2}(\mathbf{1} \otimes h_{J_1})(x-y) \Big|^2  \\
&\quad\quad \times \Big(\frac{t_1}{t_1 + |y_1|}\Big)^{n \lambda_1} \Big(\frac{t_2}{t_2 + |y_2|}\Big)^{m \lambda_2} \frac{dy_1}{t_1^n} \frac{dy_2}{t_2^m} \frac{dx_1 dt_1}{t_1}\bigg)^{1/2}\bigg]^2 \frac{dx_2 dt_2}{t_2}\\
&\lesssim \sum_{J_2} \Big[ \sum_{J_1:\ell(J_1) < \ell(J_2)} A_{J_1 J_2} \big\| f_{J_1} \big\|_{L^2(\Rn)} \Big]^2
\lesssim \sum_{J_1} \big\| f_{J_1} \big\|_{L^2(\Rn)}^2 \lesssim \big\| f \big\|_{L^2(\R^{n+m})}^2.
\end{align*}
So far, we have completed the estimate of $\mathcal{G}_{\geq,<}$.

As for the term $\mathcal{G}_{<,\geq}$, it is completely symmetric with the term $\mathcal{G}_{\geq,<}$. It is worth noting that the mixed H\"{o}lder and size estimate and the combination of Carleson and H\"{o}lder estimate are symmetric, respectively. Thus the estimate for $\mathcal{G}_{<,\geq}$ is also true and we here omit its proof.

\section{The Case : $\ell(I_1) \geq \ell(I_2)$ and $\ell(J_1) \geq \ell(J_2)$.}
Similar as what we have done before, the summation $\ell(I_1) \geq \ell(I_2)$ was decomposed into the separated, nested and adjacent terms.
A similar splitting in the summation $\ell(J_1) \geq \ell(J_2)$ is also performed. This splits the whole summation into nine parts as follows.
\begin{align*}
\mathcal{G}_{\geq,\geq}
&\lesssim \mathcal{G}_{sep,sep} + \mathcal{G}_{sep,nes} + \mathcal{G}_{sep,adj} + \mathcal{G}_{nes,sep} + \mathcal{G}_{nes,nes} \\
&\quad  + \mathcal{G}_{nes,adj} + \mathcal{G}_{adj,sep} + \mathcal{G}_{adj,nes} + \mathcal{G}_{adj,adj}.
\end{align*}

\subsection{Nested/Nested : $\mathcal{G}_{nes,nes}$.}
We begin with the term $\mathcal{G}_{nes,nes}$, where the new bi-parameter phenomena will appear. Note that although this is only one of the many cases one needs to discuss in order to obtain a full estimate for $\mathcal{G}_{\geq,\geq}$ term, all the main difficulties in other cases are in fact already embedded in Nested/Nested. The fact will become more and more clear throughout the proof. Similarly, for the singular integral operators including bi-parameter and multi-parameter cases,
the Nested part is also the most difficult one. Because it involves in some paraproduct estimates and all the BMO type estimates.

The decomposition of $h_{I^{(k)}}$ in $(\ref{h-I-k})$ gives that
$$\mathcal{G}_{nes,nes}
\lesssim \mathcal{G}_{mod,mod} + \mathcal{G}_{Car,Car} + \mathcal{G}_{mod,Car} + \mathcal{G}_{Car,mod} \ ,$$
where
\begin{align*}
\mathcal{G}_{mod,mod}
&= \sum_{I,J: good} \iint_{W_{J}} \iint_{W_{I}} \iint_{\R^{n+m}} \Big| \sum_{k=1}^\infty \sum_{i=1}^\infty
f_{I^{(k)} J^{(i)}} \theta_{t_1,t_2}(s_I^k \otimes s_J^{i})(x-y) \Big|^2 \\
&\quad\quad\quad\quad  \times \Big(\frac{t_1}{t_1 + |y_1|}\Big)^{n \lambda_1} \Big(\frac{t_2}{t_2 + |y_2|}\Big)^{m \lambda_2} \frac{dy_1}{t_1^n} \frac{dy_2}{t_2^m} \frac{dx_1 dt_1}{t_1} \frac{dx_2 dt_2}{t_2},
\end{align*}
\begin{align*}
\mathcal{G}_{mod,Car}
&= \sum_{I,J: good} \iint_{W_{J}} \iint_{W_{I}} \iint_{\R^{n+m}} \Big| \sum_{k=1}^\infty
\langle f_{I^{(k)}} \rangle_J \theta_{t_1,t_2}(s_I^k \otimes \mathbf{1})(x-y) \Big|^2 \\
&\quad\quad\quad\quad  \times \Big(\frac{t_1}{t_1 + |y_1|}\Big)^{n \lambda_1} \Big(\frac{t_2}{t_2 + |y_2|}\Big)^{m \lambda_2} \frac{dy_1}{t_1^n} \frac{dy_2}{t_2^m} \frac{dx_1 dt_1}{t_1} \frac{dx_2 dt_2}{t_2},
\end{align*}
\begin{align*}
\mathcal{G}_{Car,mod}
&= \sum_{I,J: good} \iint_{W_{J}} \iint_{W_{I}} \iint_{\R^{n+m}} \Big| \sum_{\ell=1}^\infty
\langle f_{J^{(i)}} \rangle_I \theta_{t_1,t_2}(\mathbf{1} \otimes s_J^i)(x-y) \Big|^2 \\
&\quad\quad\quad\quad  \times \Big(\frac{t_1}{t_1 + |y_1|}\Big)^{n \lambda_1} \Big(\frac{t_2}{t_2 + |y_2|}\Big)^{m \lambda_2} \frac{dy_1}{t_1^n} \frac{dy_2}{t_2^m} \frac{dx_1 dt_1}{t_1} \frac{dx_2 dt_2}{t_2},
\end{align*}
and
\begin{align*}
\mathcal{G}_{Car,Car}
&= \sum_{I,J: good} |\langle f \rangle_{I \times J}|^2 \iint_{W_{J}} \iint_{W_{I}} \iint_{\R^{n+m}} | \theta_{t_1,t_2} (\mathbf{1})(x-y) \Big|^2 \\
&\quad\quad \times \Big(\frac{t_1}{t_1 + |y_1|}\Big)^{n \lambda_1} \Big(\frac{t_2}{t_2 + |y_2|}\Big)^{m \lambda_2} \frac{dy_1}{t_1^n} \frac{dy_2}{t_2^m} \frac{dx_1 dt_1}{t_1} \frac{dx_2 dt_2}{t_2}.
\end{align*}
\subsubsection{\bf{Estimate of} $\mathcal{G}_{mod,mod}$.}
We proceed using the standard argument as in Lemma $\ref{2-alpha-k-2}$. The size condition and $(\ref{estimate-2})$ lead to the bound
\begin{align*}
\bigg( \iint_{\R^{n+m}} |\theta_{t_1,t_2}(s_I^k \otimes s_J^{i})(x-y)|^2
\Big(\frac{t_1}{t_1 + |y_1|}\Big)^{n \lambda_1}& \Big(\frac{t_2}{t_2 + |y_2|}\Big)^{m \lambda_2} \frac{dy_1}{t_1^n}
\frac{dy_2}{t_2^m} \bigg)^{1/2} \\
&\lesssim 2^{-\alpha k/2} |I^{(k)}|^{-1/2} \cdot 2^{-\beta i} |J^{(i)}|^{-1/2}.
\end{align*}
It is similar to estimate $\mathcal{G}_{mod,<}$ to analyze $\mathcal{G}_{mod,mod}$.
\begin{align*}
\mathcal{G}_{mod,mod}
&\lesssim \sum_{k,i} 2^{-\alpha k/2} 2^{-\beta i} \sum_{Q,R}|f_{QR}|^2 \frac{1}{|Q|} \sum_{I:I^{(k)}=Q} |I| \cdot \frac{1}{|R|} \sum_{J:J^{(i)}=R} |J| \\
&\lesssim \big\| f \big\|_{L^2(\R^{n+m})}^2 .
\end{align*}
\subsubsection{\bf{Estimate of} $\mathcal{G}_{Car,Car}$.}
Applying the bi-parameter Carleson condition, it immediately yields that
\begin{align*}
\mathcal{G}_{Car,Car}
&=\sum_{I,J}|\langle f \rangle_{I \times J}|^2 C_{I J}^{\mathcal{D}}
= 2 \int_{0}^{\infty} \sum_{\substack{I,J \\ |\langle f \rangle_{I \times J}| > t}} C_{I J}^{\mathcal{D}} t \ dt \\
&\lesssim \int_{0}^{\infty} \sum_{\substack{I,J \\ I \times J \subset \{M_{\mathcal{D}} f > t\}}} C_{I J}^{\mathcal{D}} t \ dt
\lesssim \int_{0}^{\infty} |\{M_{\mathcal{D}} f > t\}| t \ dt \\
&\lesssim \big\| M_{\mathcal{D}} f \big\|_{L^2(\R^{n+m})}^2
\lesssim \big\| f \big\|_{L^2(\R^{n+m})}^2,
\end{align*}
where in the last step we have used the $L^p(1<p<\infty)$ boundedness of the strong maximal function associated with rectangles.
\subsubsection{\bf{Estimate of} $\mathcal{G}_{Car,mod}$ and $\mathcal{G}_{mod,Car}$.}

\begin{lemma}\label{Car-2}
Let $J \in \mathcal{D}_{m,good}$, $(x_2,t_2) \in W_{J_2}$ and $i \in \N$ be fixed. Then the Carleson condition is satisfied
\begin{align*}
\sum_{I' \subset I} \iint_{W_{I'}} \iint_{\R^{n+m}} |\theta_{t_1,t_2}(\mathbf{1} \otimes s_J^{i})(x-y)|^2 \Big(\frac{t_1}{t_1 + |y_1|}\Big)^{n \lambda_1} &\Big(\frac{t_2}{t_2 + |y_2|}\Big)^{m \lambda_2} \frac{dy_1}{t_1^n} \frac{dy_2}{t_2^m} \frac{dx_1 dt_1}{t_1} \\
&\quad \lesssim 2^{-\beta i}|I| \cdot |J^{(i)}|^{-1}.
\end{align*}
\end{lemma}

The proof of Lemma $\ref{Car-2}$ is similar to Lemma $\ref{Car-1}$. The size condition and mixed Carleson and size estimate are used. In addition, the inequality $(\ref{estimate-2})$ is used twice.

\qed

Therefore, $\mathcal{G}_{Car,mod}$ is bounded as below.
\begin{align*}
\mathcal{G}_{Car,mod}
&\leq \sum_{J:good} \iint_{W_{J}} \sum_{I} \bigg[\sum_{i=1}^\infty |\langle f_{J^{(i)}}\rangle_I| \bigg( \iint_{W_I} \iint_{\R^{n+m}} |\theta_{t_1,t_2}(\mathbf{1} \otimes s_J^{i})(x-y)|^2 \\
&\quad\quad \times \Big(\frac{t_1}{t_1 + |y_1|}\Big)^{n \lambda_1} \Big(\frac{t_2}{t_2 + |y_2|}\Big)^{m \lambda_2} \frac{dy_1}{t_1^n} \frac{dy_2}{t_2^m} \frac{dx_1 dt_1}{t_1} \bigg)^{1/2} \bigg]^2 \frac{dx_2 dt_2}{t_2} \\
&\lesssim \sum_{i=1}^\infty 2^{-\beta i/2} \sum_{R} \big\| f_R \big\|_{L^2(\Rn)}^2 \frac{1}{|R|} \sum_{J:J^{(i)=R}}|I|
\lesssim \big\| f \big\|_{L^2(\R^{n+m})}^2 .
\end{align*}
\subsection{The rest of terms.}
As for the estimates of the remaining terms, they are simply combinations of the techniques we have used above. Thereby, we here only present certain key points.

When reviewing the above proof, one will realize that the central part is to dominate $\mathcal{P}(x,t)$, $\mathcal{Q}(x,t)$ and $\mathcal{R}(x_2,t_2)$. So do the rest of terms. Moreover, the initial estimates of $\mathcal{P}$, $\mathcal{Q}$ and $\mathcal{R}$ are retained in the inequality $(\ref{iint-estimate})$, Lemma $\ref{2-alpha-k-2}$ and Lemma $\ref{Car-1}$ respectively. They do not involve the relationship of side length of cubes $I_1$, $I_2$, $J_1$ and $J_2$. Thus, based on
the inequality $(\ref{iint-estimate})$, Lemma $\ref{2-alpha-k-2}$ and Lemma $\ref{Car-1}$, one only needs to add the corresponding the relationship of side length.

Consequently, using the size condition or the mixed H\"{o}lder and size condition, it yields the bounds for $\mathcal{G}_{sep,sep}$, $\mathcal{G}_{sep,adj}$, $\mathcal{G}_{adj,adj}$ and $\mathcal{G}_{adj,sep}$ directly. Finally, for the terms $\mathcal{G}_{nes,sep}$ and $\mathcal{G}_{nes,adj}$, $nes$ is split into $mod$ and $Car$. Applying the size condition and the combinations of Carleson and size estimate, we will bound them. The terms $\mathcal{G}_{sep,nes}$ and $\mathcal{G}_{adj,nes}$ are symmetric with respect to them respectively.

\qed
\section{The Necessity of Bi-parameter Carleson Condition}\label{Sec-necessity}
We here show that the bi-parameter Carleson condition is necessary for $g_{\lambda_1,\lambda_2}^*$-function to be bounded on $L^2(\R^{n+m})$.

Suppose that $\theta_{t_1,t_2}=\theta_{t_1}^n \otimes \theta_{t_2}^m$ is bounded on $L^2(\R^{n+m})$, where $\theta_{t_1}^n$ has a kernel $s^n_{t_1}(x_1, y_1)$, $\theta^n_{t_2}$ has a kernel $s^m_{t_2}(x_2, y_2)$, $x_1, y_1 \in \Rn$, $x_2, y_2\in \R^m$, $t_1, t_2 > 0$. We assume that these satisfy the size condition and the corresponding $L^2$ bounds in $\Rn$ and $\R^m$. We shall show that the bi-parameter Carleson condition $(\ref{Car-condition})$ holds.

Define $\widetilde{\Omega}=\{ M_{\mathcal{D}}\mathbf{1}_{\Omega} > 1/2 \}$ and
$\widehat{\Omega}= \{ M \mathbf{1}_{\widetilde{\Omega}} > c \}$ for a small enough dimensional constant $c =c(n, m)$, where
$M_{\mathcal{D}}$ denote the strong maximal function related to the grid $\mathcal{D}$ and $M$ denote the strong maximal function. From the endpoint estimates for $M$ and $M_{\mathcal{D}}$, it follows that
$|\widehat{\Omega}| \lesssim |\widetilde{\Omega}| \lesssim |\Omega|$. Hence, it is enough to show that
\begin{align*}
\sum_{\substack{I \times J \in \mathcal{D} \\ I \times J \subset \Omega}}
&\iint_{W_J} \iint_{W_I} \iint_{\R^{n+m}}|\theta_{t_1,t_2} \mathbf{1}_{\widehat{\Omega}^c}(y_1, y_2)|^2
\Big(\frac{t_1}{t_1 + |x_1 - y_1|}\Big)^{n \lambda_1} \\
&\quad \times \Big(\frac{t_2}{t_2 + |x_2 - y_2|}\Big)^{m \lambda_2} \frac{dy_1}{t_1^n}
\frac{dy_2}{t_2^m} \frac{dx_1 dt_1}{t_1} \frac{dx_2 dt_2}{t_2}
\lesssim |\Omega| .
\end{align*}

For every $J \in \mathcal{D}_m$ we let $\mathcal{F}_J$ consist of the maximal $F \in \mathcal{D}_n$ for which
$F \times J \subset \widetilde{\Omega}$. Then we define $F_J:= \bigcup_{F \in \mathcal{F}_J} 2F$. Moreover, for fixed $I \in \mathcal{D}_n$, let $\mathcal{G}_I$ be the family of the maximal
$G \in \mathcal{D}_m$ for which $I \times G \subset \Omega$, and $I_G \in \mathcal{D}_n$ be the maximal cube for which $I_G \supset I$ and $I_G \times G \subset \widetilde{\Omega}$. So, we only need to show the following inequalities.
\begin{align*}
\mathcal{G}_1
&=\sum_{\substack{I \times J \in \mathcal{D} \\ I \times J \subset \Omega}}
\iint_{W_J} \iint_{W_I} \iint_{\R^{n+m}}|\theta_{t_1,t_2} (\mathbf{1}_{\widehat{\Omega}^c} \mathbf{1}_{F_J})(y_1, y_2)|^2
\Big(\frac{t_1}{t_1 + |x_1 - y_1|}\Big)^{n \lambda_1} \\
&\quad\quad \times \Big(\frac{t_2}{t_2 + |x_2 - y_2|}\Big)^{m \lambda_2} \frac{dy_1}{t_1^n}
\frac{dy_2}{t_2^m} \frac{dx_1 dt_1}{t_1} \frac{dx_2 dt_2}{t_2} \\
&:= \sum_{J} \iint_{W_J} \mathcal{G}_J(x_2,t_2)\frac{dx_2 dt_2}{t_2}
\lesssim |\Omega|,
\end{align*}
and

\begin{align*}
\mathcal{G}_2
&=\sum_{\substack{I \times J \in \mathcal{D} \\ I \times J \subset \Omega}}
\iint_{W_J} \iint_{W_I} \iint_{\R^{n+m}}|\theta_{t_1,t_2} (\mathbf{1}_{\widehat{\Omega}^c} \mathbf{1}_{F_J^c})(y_1, y_2)|^2
\Big(\frac{t_1}{t_1 + |x_1 - y_1|}\Big)^{n \lambda_1} \\
&\quad\quad \times \Big(\frac{t_2}{t_2 + |x_2 - y_2|}\Big)^{m \lambda_2} \frac{dy_1}{t_1^n}
\frac{dy_2}{t_2^m} \frac{dx_1 dt_1}{t_1} \frac{dx_2 dt_2}{t_2}\\
&:= \sum_{I} \iint_{W_I} \mathcal{G}_I(x_1,t_1)\frac{dx_1 dt_1}{t_1}
\lesssim |\Omega|.
\end{align*}
To attain the goal, we need to first bound $\mathcal{G}_J(x_2,t_2)$ and $\mathcal{G}_I(x_1,t_1)$. Actually, Minkowski's integral inequlity and size estimate yield that

\begin{align*}
\mathcal{G}_J(x_2,t_2)
&\lesssim \bigg[ \int_{\R^m} \bigg( \iint_{\R^{n+1}_{+}} \iint_{\R^{n+m}} |K_{t_2}^m(y_2,z_2)|^2 |\theta_{t_1,t_2} (\mathbf{1}_{\widehat{\Omega}^c} \mathbf{1}_{F_J})(y_1)|^2 \Big(\frac{t_1}{t_1 + |x_1 - y_1|}\Big)^{n \lambda_1} \\
&\quad\quad\quad\quad \times \Big(\frac{t_2}{t_2 + |x_2 - y_2|}\Big)^{m \lambda_2} \frac{dy_1}{t_1^n}
\frac{dy_2}{t_2^m} \frac{dx_1 dt_1}{t_1} \bigg)^{1/2} dz_2 \bigg]^2\\
&\lesssim \bigg[\int_{\R^m} \bigg( \int_{\R^m} \Big( \frac{t_2^\beta}{(t_2 + |y_2 - z_2|)^{m+\beta}} \Big)^2 \Big(\frac{t_2}{t_2 + |x_2 - y_2|}\Big)^{m \lambda_2} \frac{dy_2}{t_2^m} \bigg)^{1/2}\\&\quad\quad\times
\big\| \mathbf{1}_{\widehat{\Omega}^c}(\cdot,z_2) \mathbf{1}_{F_J} \big\|_{L^2(\Rn)} dz_2 \bigg]^2 \\
&\lesssim \bigg[\int_{\R^m}  \frac{\ell(J)^\beta}{(\ell(J)+|x_2 - z_2|)^{m+\beta}} \big\| \mathbf{1}_{\widehat{\Omega}^c}(\cdot,z_2) \mathbf{1}_{F_J} \big\|_{L^2(\Rn)} dz_2 \bigg]^2 \\
&\lesssim \int_{\R^m}  \frac{\ell(J)^\beta}{(\ell(J)+|x_2 - z_2|)^{m+\beta}} \big\| \mathbf{1}_{\widehat{\Omega}^c}(\cdot,z_2) \mathbf{1}_{F_J} \big\|_{L^2(\Rn)}^2 dz_2 \\
&\lesssim \int_{\R^m}  \frac{\ell(J)^\beta}{|z_2 - c_J|^{m+\beta}} \big\| \mathbf{1}_{\widehat{\Omega}^c}(\cdot,z_2) \mathbf{1}_{F_J} \big\|_{L^2(\Rn)}^2 dz_2 \\
&=\int_{\Rn} \mathbf{1}_{F_J}(z_1) \int_{\R^m} \frac{\ell(J)^\beta}{|z_2 - c_J|^{m+\beta}} \mathbf{1}_{\widehat{\Omega}^c}(z_1,z_2) dz_2 \ dz_1.
\end{align*}
Similarly, we may estimate
\begin{align*}
\mathcal{G}_I(x_1,t_1)
&=\sum_{G \in \mathcal{G}_I} \sum_{J:J \subset G} \iint_{W_J} \iint_{\R^{n+m}} |\theta_{t_1,t_2} (\mathbf{1}_{\widehat{\Omega}^c} \mathbf{1}_{F_J^c})(y_1,y_2)|^2
\Big(\frac{t_1}{t_1 + |x_1 - y_1|}\Big)^{n \lambda_1} \\
&\quad\quad \times \Big(\frac{t_2}{t_2 + |x_2 - y_2|}\Big)^{m \lambda_2} \frac{dy_1}{t_1^n}
\frac{dy_2}{t_2^m} \frac{dx_2 dt_2}{t_2} \\
&\lesssim \bigg[\int_{\R^n}  \frac{\ell(I)^\alpha}{(\ell(I)+|x_1 - z_1|)^{n+\alpha}}
\Big( \sum_{G \in \mathcal{G}_I} \mathbf{1}_{(2 I_G)^c}(z_1) |G| \Big)^{1/2} dz_1 \bigg]^2 \\
&\lesssim \sum_{G \in \mathcal{G}_I} |G| \int_{\R^n}  \frac{\ell(I)^\alpha}{(\ell(I)+|x_1 - z_1|)^{n+\alpha}}
\mathbf{1}_{(2 I_G)^c}(z_1) dz_1 \\
\end{align*}
\begin{align*}
&\lesssim \sum_{G \in \mathcal{G}_I} |G| \int_{I_G^c} \frac{\ell(I)^\alpha}{|z_1 - c_{I_G}|^{n+\alpha}} dz_1
\lesssim \sum_{G \in \mathcal{G}_I} |G| \Big( \frac{\ell(I)}{\ell(I_G)} \Big)^\alpha.
\end{align*}
The remaining calculation is a routine application of the idea of \cite{M2014}. We here omit the details. Finally, we obtain
\begin{align*}
\mathcal{G}_1 = \sum_{J} \iint_{W_J} \mathcal{G}_J(x_2,t_2)\frac{dx_2 dt_2}{t_2} \lesssim |\Omega|, \ \ \
\mathcal{G}_2 = \sum_{I} \iint_{W_I} \mathcal{G}_I(x_1,t_1)\frac{dx_1 dt_1}{t_1} \lesssim |\Omega|.
\end{align*}
Thus, we have proved the necessity.
\qed


\end{document}